\documentclass[11pt,a4paper]{amsart}

\usepackage{amsaddr}
\usepackage{hyperref}

\usepackage{amsmath,amssymb,amsfonts,euscript,graphicx,amsthm}
\usepackage[all]{xy}
\usepackage{enumerate}
\usepackage{a4wide}

\usepackage{xcolor}

\DeclareMathOperator{\Char}{char}
\DeclareMathOperator{\id}{id}
\DeclareMathOperator{\Ann}{Ann}

  \newtheorem{The}{Theorem}[section]
  \newtheorem{Pro}[The]{Proposition}
  \newtheorem{Lem}[The]{Lemma}
  \newtheorem{Cor}[The]{Corollary}
	\theoremstyle{definition}

  \newtheorem{Examp}[The]{Example}
	
	\newtheorem{Ques}{Question}

	\theoremstyle{remark}
	\newtheorem{Rem}[The]{Remark}

\newcommand{\Z}{\mathbb{Z}}

\begin{document}

\author{Samaneh Baghdari}
\address{
Department of Mathematics,
Isfahan University of Technology,\\
Isfahan 81746-73441, Iran}
\email{s.baghdari@math.iut.ac.ir}

\author{Johan \"{O}inert}
\address{
Department of Mathematics and Natural Sciences,
Blekinge Institute of Technology,
SE-37179 Karlskrona, Sweden}
\email{johan.oinert@bth.se}

\makeatletter
\@namedef{subjclassname@2020}{%
  \textup{2020} Mathematics Subject Classification}
\makeatother

\subjclass[2020]{16S34, 16D70, 16P20}
\keywords{Group ring, K\"{o}the ring, Pure semisimple}

\thanks{The authors thank Mahmood Behboodi for fruitful discussions.}

\title[Pure semisimple and K\"{o}the group rings]{Pure semisimple and K\"{o}the group rings}

\date{\today}

\begin{abstract}
	In this article we provide a complete characterization of abelian group rings which are K\"{o}the rings. We also provide characterizations of (possibly non-abelian) group rings over division rings which are K\"{o}the rings, both in characteristic zero and in prime characteristic, and prove a Maschke type result for pure semisimplicity of group rings. 
	Furthermore, we illustrate our results by several examples.
\end{abstract}

\maketitle

\section{Introduction}

Throughout this article, all rings are assumed to be associative.
Recall that nowadays a unital ring $S$ is said to be a \emph{left (resp. right) K\"{o}the ring}
if every left (resp. right) $S$-module is a direct sum of cyclic submodules.
If $S$ is both a left and a right K\"{o}the ring, then $S$ is simply called a \emph{K\"{o}the ring}.
If $S$ is both a left (resp. right) artinian ring and a left (resp. right) principal ideal ring,
then we say that $S$ is a \emph{left (resp. right) artinian principal ideal ring}.
If $S$ is both a left and a right artinian principal ideal ring, then $S$ is simple called an \emph{artinian principal ideal ring}.
A unital ring $S$ is said to be a \emph{left (resp. right) pure semisimple ring} if every left (resp. right) $S$-module is a direct sum of finitely generated modules.

In \cite{Kothe} Gottfried K\"{o}the
proved the following result.

\begin{The}[K\"{o}the]\label{thm:Kothe}
Let $S$ be a unital ring.
If $S$ is an artinian principal ideal ring,
then $S$ is a K\"{o}the ring.
\end{The}

K\"{o}the \cite{Kothe} also showed that if $S$ is a unital commutative artinian K\"{o}the ring,
then $S$ is necessarily a principal ideal ring.
Later, Cohen and Kaplansky \cite{CohenKap} showed that
if $S$ is a unital commutative K\"{o}the ring,
then $S$ is necessarily an artinian principal ideal ring.
By combining K\"{o}the's results with those of Cohen and Kaplansky \cite{CohenKap}
one obtains the following characterization in the commutative setting.

\begin{The}[K\"{o}the, Cohen \& Kaplansky]\label{thm:KCP}
Let $S$ be a unital commutative ring.
Then $S$ is a K\"{o}the ring
if and only if
$S$ is an artinian principal ideal ring. 
\end{The}

Although a (possibly non-commutative) K\"{o}the ring is necessarily artinian
(see Proposition~\ref{prop:FaithWalker}),
the above characterization does not hold in the non-commutative setting.
Indeed, in \cite{Nakayama} Nakayama gave an example of a non-commutative K\"{o}the ring which fails to be a principal ideal ring.

Nevertheless,
as was shown by 
Behboodi et al. \cite{BehboodiEtAl},
Theorem~\ref{thm:KCP}
can in fact be generalized to certain (potentially) non-commutative rings, namely the abelian rings.

\begin{The}[Behboodi et al.]\label{thm:Behboodi}
Let $S$ be a unital abelian ring, i.e. a unital ring in which every idempotent is central.
Then $S$ is a K\"{o}the ring
if and only if
$S$ is an artinian principal ideal ring.
\end{The}

For further reading on the investigation of K\"{o}the rings, we refer the reader to \cite{BehboodiEtAl,Faith,FazelNasIsf,Girvan1973,Griffith}.

Recall that given a unital ring $R$ and a group $G$, the group ring $R[G]$ is a free left $R$-module
with (a copy of) $G$ as its basis and with multiplication defined as the bilinear extension of the rule
$r_1 g_1 \cdot r_2 g_2 = r_1r_2 g_1g_2$, for $r_1,r_2\in R$ and $g_1,g_2\in G$.
For an excellent introduction to the theory of group rings, we refer the reader to Passman's extensive
book \cite{PassmanBook}, and the references therein.
For examples of applications in coding theory,
we refer the reader to \cite{JitmanEtAl}
which examines certain codes based on commutative K\"{o}the group rings.

In this article we will attempt to answer the following questions.

\begin{Ques}\label{Q:main}
Given a ring $R$ and a group $G$, when is the group ring $R[G]$ a K\"{o}the ring?
\end{Ques}

\begin{Ques}\label{Q:mainPure}
Given a ring $R$ and a group $G$, when is the group ring $R[G]$ a pure semisimple ring?
\end{Ques}

Here is an outline of this article.

In Section~\ref{Sec:NotaPrel} we gather some definitions and well-known facts from group ring theory
that we need in the sequel.
In Section~\ref{Sec:General} we make some useful observations on general rings.
We also obtain necessary conditions for a group ring to be a K\"{o}the ring or pure semisimple (see Proposition~\ref{prop:KoetheNec}).
In Section~\ref{Sec:local} we prepare for Section~\ref{Sec:MainResult}
by establishing a characterization of abelian K\"{o}the group rings
over local rings (see Theorem~\ref{thm:localcase}).
In Section~\ref{Sec:MainResult}
we completely answer Question~\ref{Q:main} for abelian group rings,
by proving our first main result.

\begin{The}\label{thm:MainThmA}
Let $R$ be a unital ring, let $G$ be a group, and suppose that $R[G]$ is abelian. 
The following two assertions are equivalent:
\begin{enumerate}[{\rm (i)}]
    \item The group ring $R[G]$ is a K\"{o}the ring;
    \item There is a positive integer $n$ and local rings $(R_i,M_i)$, for $i\in \{1,\ldots,n\}$,
		such that $R = R_1 \times \ldots \times R_n$, and $R$ is a K\"{o}the ring.
		The group $G$ is finite and $p^{\prime}$-by-cyclic $p$,
	for every $p \in \Z^+ \cap \{ \Char(R_i/M_i) \mid 1 \leq i \leq n\}$.
  Moreover, $|G| \cdot 1_{R_i} \in U(R_i)$
  whenever $R_i$ is not semiprimitive.
\end{enumerate}
\end{The}

Furthermore, in Section~\ref{Sec:division} we completely answer Question~\ref{Q:main} both in the case when $R$ is a division ring of characteristic zero
(see Theorem~\ref{thm:divNC}), and in the case when $R$ is a division ring of prime characteristic and $G$ is a finite lagrangian Dedekind group (see Theorem~\ref{thm:divPrime}).
In Section~\ref{Sec:PureProj} we introduce a technique involving pure projective modules 
as a means to tackle Question~\ref{Q:mainPure}
and use it to prove a Maschke type result for pure semisimplicity of group rings, which is our second main result.

\begin{The}\label{thm:RGPureSemisimple}
Let $R$ be a ring and let $G$ be a finite group.
Suppose that $|G| \cdot 1_R \in U(R)$.
Then $R$ is left (resp. right) pure semisimple if and only if the group ring $R[G]$ is left (resp. right) pure semisimple.
\end{The}

For commutative group rings, the above result allows us to give an alternative proof
of our first main result.
In Section~\ref{Sec:Examples} we present examples of group rings which are K\"{o}the rings resp. not K\"{o}the rings.

\section{Notation and preliminaries}\label{Sec:NotaPrel}

In this section we recall notation, important definitions and earlier results.

\subsection{Group rings}

Let $R$ be a unital ring and let $G$ be a multiplicatively written group.
The \emph{group ring of $G$ over $R$} is denoted by $R[G]$. Each element 
$a \in R[G]$ can be uniquely written as $a = \sum_{g \in G} r_{g} g$ 
where $r_g \in R$ is zero for all but finitely many $g\in G$.

\begin{Rem}\label{rem:AugMap}
For any unital ring $S$, we let $U(S)$ denote the group of invertible elements of $S$.
\end{Rem}

\subsection{Artinian group rings}

We will invoke the following result by Connell \cite{Connell}.

\begin{The}[Connell]\label{thm:Connell}
Let $R$ be a unital ring and let $G$ be a group.
The group ring $R[G]$ is left (resp. right) artinian if and only if $R$ is left (resp. right) artinian and $G$ is finite.
\end{The}

\subsection{Principal ideal group rings}

If $\mathcal{A}$ and $\mathcal{B}$ are two classes of groups, then we say 
that a group $G$ is \emph{$\mathcal{A}$-by-$\mathcal{B}$} if there exists a 
normal subgroup $N$ of $G$, such that $N \in \mathcal{A}$ and $ G/N \in \mathcal{B} $.
For a finite group $G$, we shall say that \emph{$G$ is a $p$-group} if the 
order of $G$ is a power of the prime number $p$, and
that \emph{$G$ is a $p^{\prime}$-group}
if the order of $G$ is relatively prime to $p$.

For finite groups, the
equivalence (ii)$\Longleftrightarrow$(iii) in the theorem below follows immediately from \cite[Theorem 3]{Dorsey} (see also \cite[Theorem 4.1]{PassmanPaper}).
Using the exact same proof technique as in \cite{Dorsey,PassmanPaper}, it is also possible to establish the equivalence (i)$\Longleftrightarrow$(iii).

\begin{The}[Passman, Dorsey]\label{thm:Passman} 
Let $K$ be a division ring, let $G$ be a finite group, and consider the group ring $K[G]$.
The following three assertions are equivalent:
\begin{enumerate}[{\rm (i)}]
	\item $K[G]$ is a left principal ideal ring;
	\item $K[G]$ is a right principal ideal ring;
	\item $\Char(K) = 0$, or\\
	$\Char(K) = p>0$ and $G$ is	$p^{\prime}$-by-cyclic $p$.
\end{enumerate}
\end{The}

\begin{Rem}\label{Rem:PrepForDorsey}
(a) Let $I$ be an ideal of a ring $S$ and let $s+I$ be an idempotent element of $S/I$.
Following \cite[p. 301]{AndFul} we say that $s+I$ can be \emph{lifted (to $u$) modulo $I$} in case there is an idempotent $u\in S$ such that $s+I=u+I$.
We say that \emph{idempotents lift modulo $I$} in case every idempotent in $S/I$ can be lifted to an idempotent in $S$.

\noindent (b) A ring $S$ is said to be \emph{local}
if $S$ has a unique maximal left ideal,
or equivalently,
if $S$ has a unique maximal right ideal.
If $S$ is a local ring, then its unique maximal left ideal coincides
with its unique maximal right ideal and with the Jacobson radical $J(S)$.
We will let $(S,M)$ denote a local ring $S$ together with its maximal ideal $M$.

\noindent (c) Recall from \cite[p. 397]{Dorsey} that, given a unital local artinian principal ideal ring $R$, a finite group $G$ is said to be \emph{$R$-admissible}
if $|G|\cdot 1_{R/J(R)} \in U(R/J(R))$ and each centrally primitive idempotent of $(R/J(R))[G]$
can be lifted to a centrally primitive idempotent of $(R/J(R)^2)[G]$.
\end{Rem}

For finite groups, the following result follows immediately from \cite[Theorem 4]{Dorsey}.

\begin{The}[Dorsey]\label{thm:DorseyFinite}
Let $(R,M)$ be a unital local artinian principal ideal ring, let $G$ be a finite group, and consider the group ring $R[G]$.
The following two assertions are equivalent:
\begin{enumerate}[{\rm (i)}]
	\item $R[G]$ is a principal ideal ring;
	\item $\Char(R/M) = 0$: If $R$ is not a division ring, then $G$ is an $R$-admissible group. \\
	$\Char(R/M) = p>0$: $G$ is $p^{\prime}$-by-cyclic $p$.
	If $R$ is not a division ring, then $G$ is an $R$-admissible group.
\end{enumerate}
\end{The}

\subsection{Pure projectivity and pure semisimplicity}

Let $S$ be a unital ring. A short exact sequence
\begin{displaymath}
\xymatrix@1{
\mathcal{E} : 0 \ar[r]^{} & A \ar[r]^{f} & B \ar[r]^{g} & C \ar[r]^{} & 0
}
\end{displaymath}
of left (resp. right) $S$-modules
is said to be \emph{pure exact}
if $D \otimes_S \mathcal{E}$ (resp. $\mathcal{E} \otimes_S D$)
is an exact sequence (of abelian groups)
for any right (resp. left) $S$-module $D$.

A submodule $P$ of a left
$S$-module $M$
is said to be a \emph{pure submodule} of $M$ if and only if the following holds:
For any $m$-by-$n$ matrix $A=(a_{ij})$ with entries in $S$,
and any set $y_1,\ldots,y_m$ of elements of $P$,
if there exist elements $x_1,\ldots,x_n$ in $M$
such that
$\sum_{j=1}^n a_{ij} x_j = y_i$ for $i \in \{1,\ldots,m\}$
then there also exist elements $x_1', \ldots, x_n'$ in $P$ such that
$\sum_{j=1}^n a_{ij} x_j' = y_i$ for $i \in \{1,\ldots,m\}$.
Pure submodules of right $S$-modules are defined analogously.

\begin{Rem}
$\mathcal{E}$ is pure exact if and only if $f(A)$ is a pure submodule of $B$.
\end{Rem}

A module is said to be \emph{pure projective} if it is projective with respect to 
pure exact sequences.
We recall the following characterization of pure semisimplicity from \cite[§53.6]{Wisbauer1991} and \cite[Theorem~13]{HuisgenZimmermann}.

\begin{Pro}\label{prop:CharPureSS}
Let $S$ be a unital ring. The following three assertions are equivalent:
\begin{enumerate}[{\rm (i)}]
	\item $S$ is left (resp. right) pure semisimple;
	\item Every left (resp. right) $S$-module is pure projective;
	\item Every left (resp. right) $S$-module is a direct sum of indecomposable modules.
\end{enumerate}
\end{Pro}

By the definitions it is clear that any left (resp. right) K\"{o}the ring is necessarily left (resp. right) pure semisimple.
As it turns out, for commutative rings the converse also holds.
Indeed, we recall the following result from \cite[Theorem 3]{Girvan1973}.

\begin{The}[Girvan]\label{thm:Girvan}
Let $S$ be a unital commutative ring.
Then $S$ is a K\"{o}the ring
if and only if
every $S$-module is pure projective.
\end{The}

\section{General observations}\label{Sec:General}

In this section we record some observations on general rings and group rings.
We begin by recalling the following result from \cite[§53.6]{Wisbauer1991}

\begin{Pro}\label{prop:FaithWalker}
Let $S$ be a unital ring.
If $S$ is a left (resp. right) pure semisimple ring, then $S$ is a left (resp. right) artinian ring.
\end{Pro}

\begin{Lem}\label{lem:HomomorphicImage}
Let $S$ and $T$ be unital rings, and suppose that $\varphi : S \to T$ is a surjective ring homomorphism.
The following three assertions hold:
\begin{enumerate}[{\rm (a)}]

	\item If $S$ is a left (resp. right) K\"{o}the ring, then $T$ is also a left (resp. right) K\"{o}the ring.
	\item If $S$ is a principal left (resp. right) ideal ring, then $T$ is also a principal left (resp. right) ideal ring.
	\item If $S$ is a left (resp. right) pure semisimple ring, then $T$ is also a left (resp. right) pure semisimple ring.
	
\end{enumerate}
\end{Lem}

\begin{proof}
(a)
Suppose that $S$ is a left K\"{o}the ring.
Let $M$ be an arbitrary left $T$-module.
By restriction of scalars, $M$ may be viewed as a left $S$-module and by assumption $M$ decomposes into a direct sum of cyclic submodules. It is not difficult to see that $M$, viewed as a left $T$-module, also decomposes into a direct sum of cyclic submodules.
Thus, $T$ is a left K\"{o}the ring.
The right-handed case is treated analogously.

(b) Suppose that $S$ is a principal left ideal ring.
Let $I$ be an arbitrary left ideal of $T$.
The set $\varphi^{-1}(I)$
is a left ideal of $S$.
By assumption, there is some $x \in S$ such that $\varphi^{-1}(I) = Sx$,
and by surjectivity of $\varphi$ we get that 
$I = \varphi(Sx) = \varphi(S) \varphi(x) = T \varphi(x)$.
This shows that $T$ is a principal left ideal ring.
The right-handed case is treated analogously.

(c) Suppose that $S$ is a left pure semisimple ring. 
We may identify $T$ with $S/\ker(\varphi)$.
Let $M$ be an arbitrary left $S/\ker(\varphi)$-module, and let
$\cdot$ denote the left scalar action 
of $S/\ker(\varphi)$ on $M$. 
By restriction of scalars, with respect to the canonical ring homomorphism $\psi : S \to S/\ker(\varphi)$, $M$ may be viewed as a left $S$-module.
We let $\star$ denote the corresponding left scalar action 
of $S$ on $M$.
By assumption (cf.~Proposition~\ref{prop:CharPureSS}), there is some index set $I$ such that
$M = \oplus_{i\in I} M_i$,
where $M_i$ is a finitely generated left $S$-module for every $i\in I$.
Note that $\ker(\varphi) \subseteq \Ann_S(M)$, since $\ker(\varphi) \star M = \psi(\ker(\varphi)) \cdot M = \{0\}$.
Thus, $M$ may be viewed as a left $S/\ker(\varphi)$-module by defining 
the left scalar action
$S/\ker(\varphi) \times M \to M, (s+\ker(\varphi),m) \mapsto (s+\ker(\varphi)) \# m := s \star m$.
Clearly, $(s+\ker(\varphi)) \# m = (s+\ker(\varphi)) \cdot m$ for every $s \in S$ and $m\in M$.
Now, using that $M_i$ is a finitely generated $S$-module, one easily verifies that $M_i$ is also a finitely generated $S/\ker(\varphi)$-module, for every $i\in I$.
By Proposition~\ref{prop:CharPureSS}, this shows that $S/\ker(\varphi)$, and hence $T$, is left pure semisimple.
The right-handed case is treated analogously.
\end{proof}

\begin{Lem}\label{lemma:AbelInj}
The following two assertions hold:
\begin{enumerate}[{\rm (a)}]
    \item Let $S$ and $T$ be unital rings, and suppose that $\varphi : S \to T$ is an injective ring homomorphism.
If $T$ is abelian, then $S$ is also abelian.
In particular, any subring of an abelian ring is abelian in itself.
    
    \item
    Let $S = \Pi_{i=1}^n S_i$ be a direct product of unital rings $S_1,\ldots,S_n$.
		Then $S$ is abelian if and only if $S_i$ is abelian for every $i\in \{1,\ldots,n\}$.
\end{enumerate}

\end{Lem}

\begin{proof}
(a)
Suppose that $T$ is abelian
and consider the subring $\varphi(S)$ of $T$.
Take an idempotent $u\in \varphi(S) \subseteq T$.
By assumption, we get that $u\in \varphi(S) \cap Z(T) \subseteq Z(\varphi(S))$.
Thus, $\varphi(S)$ is abelian.
Using that $S \cong \varphi(S)$ we conclude that $S$ is abelian.

(b)
The ''only if'' statement
follows immediately from (a) after considering the natural embedding $\iota_i : S_i \to \Pi_{i=1}^n S_i$, for each $i\in \{1,\ldots,n\}$.
Now, suppose that $S_i$ is abelian for every $i\in \{1,\ldots,n\}$.
Take an idempotent $u \in S$.
Then $u=(u_1,\ldots,u_n)$
where, by assumption, $u_i$ is a central idempotent of $S_i$, for every $i\in \{1,\ldots,n\}$.
Clearly, $u\in Z(S)$. 
This concludes the proof of the
''if'' statement.
\end{proof}

\begin{Pro}\label{prop:KoetheNec}
Let $R$ be a unital ring and let $G$ be a group.
The following two assertions hold:
\begin{enumerate}[{\rm (a)}]
	
	\item If the group ring $R[G]$ is a left (resp. right) pure semisimple ring,
	then
	$(R/I)[G/N]$ is a left (resp. right) pure semisimple ring for every proper ideal $I$ of $R$ and every normal subgroup $N$ of $G$.
	Furthermore, $(R/I)[G/N]$ is a left (resp. right) artinian ring.
	In particular, $R/I$, $R$ and $R[G]$ are
	left (resp. right) pure semisimple rings
	and
	left (resp. right) artinian rings, and $G$ is a finite group.

	\item If the group ring $R[G]$ is a left (resp. right) K\"{o}the ring,
	then
	$(R/I)[G/N]$ is a left (resp. right) K\"{o}the ring for every proper ideal $I$ of $R$ and every normal subgroup $N$ of $G$.
	Furthermore, $(R/I)[G/N]$ is a left (resp. right) artinian ring.
	In particular, $R/I$, $R$ and $R[G]$ are
	left (resp. right) K\"{o}the rings
	and
	left (resp. right) artinian rings, and $G$ is a finite group.
	
\end{enumerate}
\end{Pro}

\begin{proof}
Consider the natural quotient maps $R \to R/I, \ r \mapsto \overline{r}$
and $G \to G/N, \ g \mapsto gN$.
Define a map $\varphi : R[G] \to (R/I)[G/N]$
by naturally extending the rule $\varphi(rg) = \overline{r}gN$, for $r\in R$ and $g\in G$.
Clearly, $\varphi$ is a surjective ring homomorphism.

Most of (a) and (b) now follow immediately from Lemma~\ref{lem:HomomorphicImage}.
The claims about artinianity follow from Proposition~\ref{prop:FaithWalker},
and the finiteness of $G$ follows from Theorem~\ref{thm:Connell}.
\end{proof}

\section{Abelian group rings over local rings}\label{Sec:local}

In this section we provide a characterization of abelian K\"{o}the group rings
over local rings (see Theorem~\ref{thm:localcase}). That result will be an essential ingredient
in the proof of our first main result in Section~\ref{Sec:MainResult}.

\begin{Lem}\label{lem:InvertDivisor}
Let $R$ be a unital ring and let $n>1$ be an integer.
The following two assertions are equivalent:
\begin{enumerate}[{\rm (i)}]
    \item $n \cdot 1_R \notin U(R)$;
    \item There is a prime divisor $q$ of $n$, and a proper ideal $I$ of $R$ such that $R/I$ is a domain with $\Char(R/I)=q$.
\end{enumerate}
\end{Lem}

\begin{proof}
(ii)$\Rightarrow$(i)
Suppose that $n=q\cdot n'$ for a prime number $q$, and that $I$ is a proper ideal of $R$ such that $\Char(R/I)=q$.
Seeking a contradiction, suppose that $n \cdot 1_R \in U(R)$, 
i.e. there is some $r\in R$ such that $r\cdot n \cdot 1_R = 1_R$.
Consider the natural map $\varphi : R \to R/I$.
Clearly,
$1_R + I = \varphi(1_R)
=\varphi(r\cdot n \cdot 1_R)
=\varphi(r\cdot q \cdot n' \cdot 1_R)
= \varphi(rq)\cdot \varphi(n' \cdot 1_R)
=\varphi(r)\cdot q \cdot \varphi(n' \cdot 1_R)
=\varphi(r) \cdot I
\subseteq I.$
This is a contradiction since $I$ is proper.
The right-handed case is treated analogously.

(i)$\Rightarrow$(ii)
Suppose that $n\cdot 1_R \notin U(R)$. 
There must exist a prime number $q$, such that $n=q\cdot n'$ and such that $q \cdot 1_R \notin U(R)$.
Consider the ideal $I=R \cdot q \cdot 1_R$ of $R$. 
Clearly, $I$ must be proper since $q \cdot 1_R \notin U(R)$.
Using that $q$ is prime, it is easy to see that $R/I$ is a domain and that $\Char(R/I)=q$.
The right-handed case is treated analogously.
\end{proof}

\begin{Lem}\label{lem:invertibility}
Let $(R, M)$ be a unital local ring and let $n>1$ be an integer.
The following two assertions hold:
\begin{enumerate}[{\rm (a)}]
	\item If $\Char(R/M)=0$, then $n \cdot 1_R \in U(R)$.

	\item Suppose that $\Char(R/M)=p>0$.
	Then $n \cdot 1_R \notin U(R)$ if and only if $p$ divides $n$.
\end{enumerate}
\end{Lem}

\begin{proof}
We begin with a general observation.
If $n \cdot 1_R \notin  U(R)$, then by Lemma~\ref{lem:InvertDivisor} there is a prime divisor $q$ of $n$, and a proper ideal $I$ of $R$ such that $\Char(R/I)=q>0$.
Notice that $I \subseteq M \subseteq R$, since $M$ is maximal.
From the third isomorphism theorem we get that $R/M \cong (R/I)/(M/I)$.
Recall that $\Char(R/M)$ must divide $\Char(R/I)$.

(a)
Seeking a contradiction, suppose that $n \cdot 1_R \notin  U(R)$.
Using that $\Char(R/M)=0$ and $\Char(R/I)=q>0$, we get a contradiction.
Thus, $n \cdot 1_R \in  U(R)$.

(b)
The ''if'' statement follows immediately from Lemma~\ref{lem:InvertDivisor}.
Now we show the ''only if'' statement.
Suppose that $n \cdot 1_R \notin  U(R)$.
We use Lemma~\ref{lem:InvertDivisor} to find $q$ and $I$.
Using that $\Char(R/M)=p$ and $\Char(R/I)=q$ we get $p=q$,
and this shows that $p$ divides $n$.
\end{proof}

The following result follows immediately from \cite[Lemma 5]{Dorsey}.

\begin{Lem}\label{cor:pgroup}
Let $(R, M)$ be a unital 
local artinian principal ideal ring with $\Char(R/M)=p>0$, and let $G$ be a finite group.
If $R[G]$ is a principal ideal ring and $R$ is not a division ring, then $G$ is not a $p$-group.
\end{Lem}

In the following remark we record a number of facts that will be useful in the proof of the subsequent lemma.

\begin{Rem}\label{Rem:PrepForLemma}
(a) Recall from \cite[Proposition 27.1]{AndFul}
that if $I$ is a nil ideal of a ring $S$, then idempotents lift modulo $I$.

(b) If $S$ is a left artinian ring, then $J(S)$ is nilpotent (see e.g. \cite[Theorem 15.20]{AndFul}). In particular, $J(S)$ is nil
and hence by (a) idempotents in $S$ lift modulo $J(S)$.

(c) If $I$ is an ideal of a ring $S$ such that $S/I$ is semiprimitive, i.e. $J(S/I)=\{0\}$, then $J(S) \subseteq I$ (see e.g. \cite[Ex. 4.11]{LamBook}).

(d) If $S$ is a left artinian ring, then by (b), $J(S)$ is a nilpotent ideal and hence by \cite[Theorem 22.9]{LamBook} there is a bijective correspondence between centrally primitive idempotents of $S$ and centrally primitive idempotents of $S/(J(S)^2)$.
\end{Rem}

The proof of the following lemma is inspired by \cite[Corollary 9]{Dorsey}.

\begin{Lem}\label{lem:Gadmissible}
Let $(R,M)$ be a unital local artinian principal ideal ring, 
let $G$ be a finite group, and suppose that $R[G]$ is abelian.
If $|G| \cdot 1_R \in U(R)$, then $G$ is $R$-admissible.
\end{Lem}

\begin{proof}
Given an ideal $I$ of $R$ we will let $IG$ denote the set $I \cdot R[G]$.
Notice that $IG$ is an ideal of the group ring $R[G]$.
It is easy to see that we get two natural ring isomorphisms
\begin{equation}
\frac{R[G]}{IG}
\cong 
\left(\frac{R}{I}\right) [G]
\quad
\text{ and }
\quad
\frac{R[G]}{I^2G}
\cong
\left(\frac{R}{I^2}\right)[G].
\label{eq:LemFirstQuot}
\end{equation}
We will now direct our attention to the case $I=J(R)$.

Using that $R$ is artinian, \cite[Proposition 9(25)]{Connell} yields that $J(R) \subseteq J(R[G])$.
From this we get that $J(R)G \subseteq J(R[G])$.
Notice that $R/J(R)$ is a divison ring, since $R$ is local.
Moreover, $|G|\cdot 1_R \in U(R)$ yields $|G|\cdot 1_{R/J(R)} \in U(R/J(R))$,
and hence by Maschke's theorem we conclude that $(R/J(R))[G]$ is semisimple.
In particular, by \eqref{eq:LemFirstQuot}, $(R/J(R))[G]$ and $R[G]/(J(R)G)$ are semiprimitive.
By Remark~\ref{Rem:PrepForLemma}(c) we get that
$J(R[G]) \subseteq J(R)G$.
By combining this with the previous inclusion we get that
	$J(R[G]) = J(R)G.$ 
From this we also get that $J(R[G])^2 = J(R)^2 G$.
Thus, by invoking \eqref{eq:LemFirstQuot}, we get that
\begin{equation}
(R/(J(R))[G] \cong R[G]/(J(R)G) = R[G]/(J(R[G]))
\label{eq:LemSecondEq1}
\end{equation}
and
\begin{equation}
R/(J(R)^2)[G] \cong R[G]/(J(R)^2G) = R[G]/(J(R[G])^2).
\label{eq:LemSecondEq2}
\end{equation}
Using that $G$ is finite and that $R$ is artinian, by Theorem~\ref{thm:Connell}, $R[G]$ is artinian. 
Thus, by Remark~\ref{Rem:PrepForLemma}(b), \eqref{eq:LemSecondEq1} and the fact that $R[G]$ is abelian, every centrally primitive idempotent in
$(R/J(R))[G]$
can be lifted to a centrally primitive idempotent in $R[G]$.

Furthermore, by Remark~\ref{Rem:PrepForLemma}(d) and \eqref{eq:LemSecondEq2}
there is a bijective correspondence between
the centrally primitive idempotents of $R[G]$ and the centrally primitive idempotents of
$(R/J(R)^2)[G]$.
%
%
In conclusion,
every centrally primitive idempotent in $(R/J(R))[G]$
can be lifted to a centrally primitive idempotent in $(R/J(R)^2)[G]$.
This shows that $G$ is $R$-admissible.
\end{proof}

\begin{Rem}
By the third isomorphism theorem for rings, we get that
$R[G]/J(R[G]) \cong (R[G]/J(R[G])^2) / (J(R[G])/J(R[G])^2)$.
If we combine this with \eqref{eq:LemSecondEq1} and \eqref{eq:LemSecondEq2}, then we see that
$(R/J(R))[G]$ may be viewed as a quotient of $(R/J(R)^2)[G]$.
Hence, in this situation it makes sense to
speak of \emph{lifting idempotents from
$(R/J(R))[G]$ to $(R/J(R)^2)[G]$}.
\end{Rem}

\begin{The}\label{thm:localcase}
  Let $(R, M)$ be a unital local ring, let $G$ be a group,
  and suppose that $R[G]$ is abelian. 
  The following two assertions are equivalent:
  \begin{enumerate}[{\rm (i)}]
      \item The group ring $R[G]$ is a K\"{o}the ring;
      
      \item $\Char(R/ M) = 0$: The ring $R$ is a K\"{o}the ring and $G$ is a finite group. If $R$ is not a division ring, then $|G| \cdot 1_R \in U(R)$.
 \\ $\Char(R/ M) = p > 0$: The ring $R$ is a K\"{o}the ring and $G$ is a finite $p^{\prime}$-by-cyclic $p$ group. If $R$ is not a division ring, then $|G| \cdot 1_R \in U(R)$.
  \end{enumerate}
 \end{The}

\begin{proof}
(i)$\Rightarrow$(ii)
By Proposition~\ref{prop:KoetheNec}, $R$ is a K\"{o}the ring and $G$ is a finite group.
If $G$ is the trivial group, then clearly $|G|\cdot 1_R = 1_R \in U(R)$.
Therefore, suppose that $G$ is non-trivial.
Notice that $R/M$ is a division ring whose characteristic is either zero or a prime.

\textbf{Case 1}: $\Char(R/M) = 0$.
Consider $n=|G|>1$. By Lemma~\ref{lem:invertibility}(a) we conclude that $|G|\cdot 1_R \in U(R)$.\\ 
\textbf{Case 2}:
$\Char(R/M)= p$.
By Theorem~\ref{thm:Behboodi}, $R[G]$ is a principal ideal ring,
and by Lemma~\ref{lem:HomomorphicImage}(b) the same conclusion holds for $(R/M)[G]$.
We know that $R/M$ is a division ring, and hence Theorem~\ref{thm:Passman}
yields that $G$ is $p'$-by-cyclic $p$.
Suppose that $R$ is not a division ring. Consider $n=|G|>1$.
Seeking a contradiction, suppose that
$|G|\cdot 1_R \notin  U(R)$. 
Using that $G$ is $p'$-by-cyclic $p$,
there is a normal subgroup $N$ of $G$ such that $N$ is a $p'$-group
and $G/N$ is a cyclic $p$-group.
Notice that, since $R[G]$ is a principal ideal ring,
Lemma~\ref{lem:HomomorphicImage}(b) yields that $R[G/N]$ is also a principal ideal ring. 
By Lemma~\ref{cor:pgroup} we conclude that $G/N$ is not a $p$-group.
Thus, $G/N$ is trivial, i.e. $N=G$, and hence $G$ is a $p'$-group.
By Lemma~\ref{lem:invertibility}(b) we conclude that $p$ divides $|G|$
but this is a contradiction since $G$ is a $p'$-group.

(ii)$\Rightarrow$(i)
We consider two cases.
\\ \textbf{Case 1}: $ R $ is a division ring.
Notice that $R[G]$ is an artinian ring. Indeed, by assumption $G$ is finite and $R$ 
is artinian. 
Thus, artinianity of $R[G]$ follows from Theorem~\ref{thm:Connell}.
By Theorem~\ref{thm:Passman}, 
$R[G]$ is a principal ideal ring. Hence,
by Theorem \ref{thm:Kothe}, 
the group ring $R[G]$ is a K\"{o}the ring.
\\ \textbf{Case 2}: $ R $ is not a division ring.
By assumption, $|G| \cdot 1_R \in U(R)$.
It follows, by Lemma~\ref{lem:Gadmissible} that $G$ is $R$-admissible.
By Theorem~\ref{thm:DorseyFinite},
$R[G]$ is a principal ideal ring and the desired conclusion now follows in the same way as for Case 1.
\end{proof}

\section{Proof of the first main result}\label{Sec:MainResult}

In this section we prove our first main result 
by combining 
a couple of lemmas 
with the results of the preceding sections.

The proof of the following lemma follows by
Theorem~\ref{thm:Behboodi},
Lemma~\ref{lem:HomomorphicImage}(a) 
and
Lemma~\ref{lemma:AbelInj}.

\begin{Lem}\label{lemma:FinProdAbelKothe}
Let $S = S_1 \times \ldots \times S_n$ be a unital ring.
Then $S$ is an abelian K\"{o}the ring if and only if $S_i$ is an abelian K\"{o}the ring
for every $i \in \{1,\ldots,n\}$.
\end{Lem}

We recall the following result from \cite[Proposition 3]{Habeb}.

\begin{Lem}\label{lem:AbelianDecomp}
Let $R$ be a left (or right) artinian ring.
Then $R$ is a finite product of local rings if and only if every idempotent of $R$ is central in $R$.
\end{Lem}

We are now ready to prove our first main result.

\begin{proof}[Proof of Theorem~\ref{thm:MainThmA}]
(i)$\Rightarrow$(ii)
By Proposition~\ref{prop:KoetheNec}, $R$ is a K\"{o}the ring (and thus artinian) and $G$ is finite.
The existence of $n$ and local rings $(R_1,M_1),\ldots,(R_n,M_n)$ such that $R=R_1 \times \ldots \times R_n$
follows from Lemma~\ref{lem:AbelianDecomp}.
Take an arbitrary $i \in \{1,\ldots,n\}$.
Notice that $R_i[G]$ is a K\"{o}the ring by Lemma~\ref{lem:HomomorphicImage}(a),
and abelian by Lemma~\ref{lemma:AbelInj}(a).
By Theorem~\ref{thm:localcase}, $G$ is $p'$-by-cyclic $p$ if $\Char(R_i/M_i)>0$.
If $R_i$ is not semiprimitive, then $M_i=J(R_i) \neq \{0\}$.
In that case, $R_i$ is not a division ring and it follows from Theorem~\ref{thm:localcase} that $|G|\cdot 1_{R_i} \in U(R_i)$.
\\
(ii)$\Rightarrow$(i)
By Lemma~\ref{lem:HomomorphicImage}(a) 
$R_i$ is a local K\"{o}the ring for every $i\in \{1,\ldots,n\}$.
Take $i \in \{1,\ldots,n\}$.
By Lemma~\ref{lemma:AbelInj}(a), the group ring $R_i[G]$ is abelian,
and by Theorem~\ref{thm:localcase} it is a K\"{o}the ring.
Using Lemma~\ref{lemma:FinProdAbelKothe}, we conclude that
$R[G]=\prod_{i=1}^n R_i[G]$ is a K\"{o}the ring.
\end{proof}

\section{Group rings over division rings}\label{Sec:division}

In this section we characterize K\"{o}the group rings over division rings, both in characteristic zero (see Theorem~\ref{thm:divNC}) and in prime characteristic (see Theorem~\ref{thm:divPrime}).

\begin{The}\label{thm:divNC}
Let $R$ be a division ring with $\Char(R)=0$ and let $G$ be a group.
The following three assertions are equivalent:
\begin{enumerate}[{\rm (i)}]
    \item The group ring $R[G]$ is a left K\"{o}the ring;
    \item The group ring $R[G]$ is a right K\"{o}the ring;
    \item $G$ is a finite group.
\end{enumerate}
\end{The}

\begin{proof}
(i)$\Rightarrow$(iii)
This follows immediately from Proposition~\ref{prop:KoetheNec}(b).

(iii)$\Rightarrow$(i)
The ring $R$ is a division ring and hence artinian.
Using that $G$ is finite, Theorem~\ref{thm:Connell} yields that $R[G]$ is artinian.
By Theorem~\ref{thm:Passman}, $R[G]$ is a principal ideal ring.
The desired conclusion now follows from Theorem~\ref{thm:Kothe}.

(ii)$\Leftrightarrow$(iii)
The proof is analogous to the proof of (i)$\Leftrightarrow$(iii).
\end{proof}

We prepare ourselves for the case of prime characteristic by recalling 
the following result from \cite[Theorem(2), p.138]{Nicholson}.

\begin{The}[Nicholson]\label{thm:nicholson}
Let $R$ be a unital local ring and let $G$ be a locally finite $p$-group.
If $p \in J(R)$, then the group ring $R[G]$ is local.
\end{The}

\begin{Cor}\label{cor:CharPlocal}
Let $R$ be a division ring with $\Char(R)=p>0$ and let $H$ be a $p$-group.
If $R[H]$ is a left (or right) K\"{o}the ring, then $H$ is finite $p'$-by-cyclic $p$.
\end{Cor}

\begin{proof}
Suppose that $R[H]$ is a left K\"{o}the ring.
By Proposition~\ref{prop:KoetheNec}, $H$ is a finite $p$-group and $R[H]$ is a left artinian ring.
Using Theorem~\ref{thm:nicholson}, we get that $R[H]$ is a local ring
and hence, by Lemma~\ref{lem:AbelianDecomp}, $R[H]$ is an abelian ring.
It follows from Theorem~\ref{thm:localcase} that $H$ is a $p'$-by-cyclic $p$-group.
The right-handed case is treated analogously to the left-handed case.
\end{proof}

Recall that a finite group $G$ is said to be \emph{lagrangian} (see e.g. \cite{HumphreysJohnson,McLain}) if for every factor of $|G|$,
$G$ possesses a subgroup of that order.
Also recall that $G$ is said to be a \emph{Dedekind group} if every subgroup of $G$ is normal in $G$.

\begin{The}\label{thm:divPrime}
Let $R$ be a division ring with $\Char(R)=p>0$ and let $G$ be a (finite) lagrangian Dedekind group.
The following three assertions are equivalent:
\begin{enumerate}[{\rm (i)}]
	\item The group ring $R[G]$ is a left K\"{o}the ring;
	\item The group ring $R[G]$ is a right K\"{o}the ring;
	\item $G$ is $p'$-by-cyclic $p$, or $R[G]$ is semisimple.
\end{enumerate}
\end{The}

\begin{proof}
(i)$\Rightarrow$(iii)
We consider two mutually exclusive cases.\\
\textbf{Case 1}: $|G| \cdot 1_R \in U(R)$.
By Maschke's theorem, $R[G]$ is semisimple.\\ 
\textbf{Case 2}: $|G| \cdot 1_R \notin U(R)$.
By Lemma~\ref{lem:invertibility}(b), we get that $p$ divides $|G|$.
There is a positive integer $r$ such that either $|G|=p^r$ or $|G|=p^r m$, where $p$ and $m$ are relatively prime.

If $|G|=p^r$, then $G$ is $p'$-by-cyclic $p$, by Corollary~\ref{cor:CharPlocal}.
On the other hand, if $|G|=p^r m$, then using that $G$ is lagrangian and Dedekind, we may choose a normal subgroup
$N$ of $G$ such that $|N|=m$. Note that $N$ is a $p'$-group and that $G/N$ is a $p$-group.
By Proposition~\ref{prop:KoetheNec} we get that $R[G/N]$ is a 
left 
K\"{o}the ring,
and hence, by Corollary~\ref{cor:CharPlocal}, $G/N$ is $p'$-by-cyclic $p$.
But the only normal subgroup of
$G/N$ which is a $p'$-group is the trivial group. Hence, $G/N$ is a cyclic $p$-group.
This shows that $G$ is $p'$-by-cyclic $p$.

(iii)$\Rightarrow$(i)
This follows by combining Theorem~\ref{thm:Connell}, Theorem~\ref{thm:Passman} and Theorem~\ref{thm:Kothe}, or by the fact that every semisimple ring, viewed as a module over itself, is a direct sum of simple, and thus cyclic, submodules.

(ii)$\Leftrightarrow$(iii)
The proof is analogous to the proof of (i)$\Leftrightarrow$(iii).
\end{proof}

\section{Pure semisimplicity of group rings}\label{Sec:PureProj}

In Section~\ref{subsec:pureproj} we make an observation about pure projectivity of modules over (not necessarily commutative) group rings and prove our second main result, which is a Maschke type result for pure semisimplicity of group rings (see Theorem~\ref{thm:RGPureSemisimple}).
In Section~\ref{subsec:commGPrings} we use our previous observation
to provide an alternative proof of Theorem~\ref{thm:localcase}
in the commutative setting.
We also rephrase our first main result in the commutative setting (see Corollary~\ref{Cor:MainThmComm}).

\subsection{Pure projectivity and pure semisimplicity}\label{subsec:pureproj}

We want to emphasize that in this subsection we are dealing with general rings that are not necessarily abelian nor commutative.

\begin{Lem}\label{lem:PureExtSeqRestriction}
Let $S$ be a ring and $T\subseteq S$ a subring.
If
\begin{equation}\label{Eq:LemPureExact1}
\xymatrix@1{
0 \ar[r]^{} & A \ar[r]^{f} & B \ar[r]^{g} & C \ar[r]^{} & 0
}
\end{equation}
is a pure exact sequence of left (resp. right) $S$-modules, then by restricion of scalars
\begin{equation}\label{Eq:LemPureExact2}
\xymatrix@1{
0 \ar[r]^{} & A  \ar[r]^{f'} & B  \ar[r]^{g'} & C \ar[r]^{} & 0
}
\end{equation}
is a pure exact sequence of left (resp. right) $T$-modules.
\end{Lem}

\begin{proof}
Suppose that \eqref{Eq:LemPureExact1} is pure exact.
It is clear that \eqref{Eq:LemPureExact2} is a short exact sequence of left (resp. right) $T$-modules.
By assumption, $f(A)$ is a pure submodule of $B$.
Using this and the fact that $T$ is a subring of $S$,
it is readily verified that $f'(A)$ is a pure submodule of $B$.
This shows that \eqref{Eq:LemPureExact2} is a pure exact sequence of left (resp. right) $T$-modules.
\end{proof}

We are now ready to prove our second main result.

\begin{proof}[Proof of Theorem~\ref{thm:RGPureSemisimple}]
The ''if'' statement follows immediately from Proposition~\ref{prop:KoetheNec}(a).
Now we show the ''only if'' statement.
Suppose that $R$ is a left pure semisimple ring.
Let $M$ be an arbitrary left $R[G]$-module and let
\begin{equation}\label{Eq:PureExactSeq}
\xymatrix@1{
0 \ar[r]^{} & K \ar[r]^{\varphi} & L \ar[r]^{\psi} & M \ar[r]^{} & 0
}
\end{equation}
be an arbitrary pure exact sequence of left $R[G]$-modules.
We are going to show that \eqref{Eq:PureExactSeq} splits, and thereby that $M$ is pure projective.
The desired conclusion will then follow from Proposition~\ref{prop:CharPureSS}.
By restriction of scalars, $K$, $L$ and $M$ may be viewed as left $R$-modules
and, by Lemma~\ref{lem:PureExtSeqRestriction},
we obtain a corresponding pure exact sequence
\begin{equation}\label{Eq:PureExactSeqR}
\xymatrix@1{
0 \ar[r]^{} & K \ar[r]^{\widetilde{\varphi}} & L \ar[r]^{\widetilde{\psi}} & M \ar[r]^{} & 0
}
\end{equation}
of left $R$-modules.
Using Proposition~\ref{prop:CharPureSS}, we conclude that \eqref{Eq:PureExactSeqR} splits. 
That is, there exists a left $R$-module homomorphism $\widetilde{\phi} : M \to L$ such that $
\widetilde{\psi}
\circ
\widetilde{\phi}  = \id_M$.

Now we define the map 
$\phi : M \to L, \ m \mapsto  |G|^{-1} \sum_{g \in G} g^{-1} \widetilde{\phi} (gm).$
It is clear that $\phi$ is well-defined and additive.
Take $h\in G$ and $m\in M$.
We get that 
\begin{eqnarray*}
\phi(hm) = 
|G|^{-1} \sum_{g \in G} g^{-1} \widetilde{\phi} (ghm) 
= |G|^{-1} \sum_{\lambda \in G} h \lambda^{-1} \widetilde{\phi} (\lambda m)
= h \left( |G|^{-1} \sum_{\lambda \in G} \lambda^{-1} \widetilde{\phi} (\lambda m) \right) 
= h \cdot \phi(m)
\end{eqnarray*}
and hence $\phi$ is a left $R[G]$-module homomorphism.
Moreover, for any $m\in M$ we get that
\begin{eqnarray*}
   \psi \circ \phi (m) &=&  \psi \left( |G|^{-1} \sum_{g \in G} g^{-1} \widetilde{\phi} (gm) \right) = |G|^{-1} \sum_{g \in G} g^{-1} \left( \psi \circ \widetilde{\phi} (gm) \right) \\
   &=&
   |G|^{-1} \sum_{g \in G} g^{-1} \left( gm \right)
   =
 \left( |G|^{-1} \sum_{g \in G} g^{-1} g \right) m
   = 1 \cdot m = m.
\end{eqnarray*}
This shows that \eqref{Eq:PureExactSeq} splits. 
The right-handed case is treated similarly.
\end{proof}

\begin{Rem}
If a group ring $R[G]$ is left (or right) pure semisimple, then $G$ is necessarily finite (see Proposition~\ref{prop:KoetheNec}).
However, $|G| \cdot 1_R \in U(R)$ is not a necessary condition for pure semisimplicity of $R[G]$.
To see this one can e.g. look at Example~\ref{ex:K2S3}.
\end{Rem}

\subsection{Commutative group rings}\label{subsec:commGPrings}

By combining Theorem~\ref{thm:RGPureSemisimple} with Theorem~\ref{thm:Girvan} we obtain the following result.

\begin{Cor}\label{Cor:RGmoduleIsPureProj}
Let $R$ be a commutative ring and let $G$ be a finite abelian group. 
Suppose that $|G|\cdot 1_R \in U(R)$.
Then $R$ is a K\"{o}the ring if and only if $R[G]$ is a K\"{o}the ring.
\end{Cor}

With the use of the above corollary we can provide an alternative proof
of Theorem~\ref{thm:localcase} for commutative group rings.
The alternative proof is identical to the proof of Theorem~\ref{thm:localcase}, except for
(ii)$\Rightarrow$(i), Case 2, where we can make use of Corollary~\ref{Cor:RGmoduleIsPureProj}
instead of invoking Theorem~\ref{thm:DorseyFinite}.

\begin{Rem}\label{Rem:CommCase}
Recall that if $R$ is a commutative artinian ring, then it has only finitely many maximal ideals
$M_1, \ldots, M_n$.
Moreover, there is a positive integer $k$ such that $\Pi_{i=1}^n M_i^k = 0$.
The ideals $M_1^k, \ldots, M_n^k$ are coprime in pairs.
Consequently one may show that
$R \cong R/M_1^k \times \ldots \times R/M_n^k$ (see e.g. \cite[Theorem 8.7]{AtiMac}).
For any $i \in \{1,\ldots,n\}$ we make the following observations:

The ring $R/M_i^k$ is a local artinian ring.
By the third isomorphism theorem we get that
$R/I \cong (R/I^k)/(I/I^k)$ for any ideal $I$ of $R$.
In particular, $M_i/M_i^k$ is a maximal ideal of $R/M_i^k$ if and only if $M_i$ is a maximal ideal of $R$.
\end{Rem}

Using the above observations, Theorem~\ref{thm:MainThmA} takes a slightly more elegant form
in the commutative setting.

\begin{Cor}\label{Cor:MainThmComm}
 Let $ R $ be a unital commutative ring and let $ G $ be an abelian group.
Denote by $\mathrm{Max}(R)$ the set of maximal ideals of $R$.
 The following two assertions are equivalent:
\begin{enumerate}[{\rm (i)}]
    \item The group ring $ R[G] $ is a K\"{o}the ring;
    \item $R$ is a K\"{o}the ring and $G$ is a finite group which is
 $p^{\prime}$-by-cyclic $p$
	for every $p \in \{ \Char(R/M) \mid M \in \mathrm{Max}(R) \}$.
  Moreover, $|G| \cdot 1_{R/M_i^k} \in U(R/M_i^k)$
  whenever $R/M_i^k$,
	appearing in the decomposition of $R$ (see Remark~\ref{Rem:CommCase}),
	is not semiprimitive.
\end{enumerate}
\end{Cor}

\section{Examples}\label{Sec:Examples}

In this section we present some examples of group rings which are K\"{o}the rings and group rings which are not.

\begin{Examp}
Let $G=Q_8$ be the quaternion group of order $8$.

(a)
Let $R=\Z$ be the ring of integers. 
By combining
\cite[Proposition 3.4]{Sehgal1975} with \cite[Lemma 6.4]{Oinert2019}
we immediately conclude that $\Z[G]$ is an abelian ring.
However, the integral group ring $\Z[G]$ is not a K\"{o}the ring for at least two reasons;
$\Z$ is not artinian, and $G$ is not $p'$-by-cyclic $p$ for any prime number $p$.

(b)
Let $R=\mathbb{H}$ be the ring of real quaternions. 
By Theorem~\ref{thm:divNC} we conclude that $\mathbb{H}[G]$ is a K\"{o}the ring.
Notice that $R$ is non-commutative and that $G$ is non-abelian.
\end{Examp}

\begin{Examp}\label{ex:K2S3}
Let $K$ be a field of characteristic $2$ and let $G=S_3$ be the symmetric group on $3$ letters.
The alternating group $A_3$ is a normal subgroup of $S_3$ with $|A_3|=3$,
and $S_3/A_3$ is cyclic of order $2$.
Thus, $S_3$ is $p'$-by-cyclic $p$ with $p=2$.
However, $K[G]$ is not abelian.
Indeed, let $e$ denote the identity permutation
and choose a non-identity permutation $g \in S_3$ satisfying $g^3=e$.
It is readily verified that the element $u=e+g+g^2$ is a non-central idempotent in $K[G]$.
Nevertheless, we notice that $K[G]$ is an artinian principal ideal ring
and hence, by Theorem~\ref{thm:Kothe}, $K[G]$ is in fact a K\"{o}the ring.
\end{Examp}

\begin{Examp}
Let $C_3$ be a cyclic group of order $3$, let $D_8$ be the dihedral group of order $8$,
and consider the finite group $G=C_3 \times D_8$.
Notice that $N=\{e_{C_3}\} \times D_8$ is a normal subgroup of $G$ with $|N|=2^3$,
and that $G/N \cong C_3$ is of order $3$.
Thus, $G$ is $p'$-by-cyclic $p$ with $p=3$.

(a)
Let $K$ be a field of characteristic $3$.
We notice that the subring $K[C_3]$ contains nonzero nilpotent elements. Hence, $K[G]$ is not reduced.
Nevertheless, by \cite[p. 69]{CoehlhoPol1987}, $K[G]$ is in fact abelian.
By Theorem~\ref{thm:MainThmA} we conclude that $K[G]$ is a K\"{o}the ring.

(b)
Let $K$ be a field of characteristic $13$.
The group ring $K[G]$ is still abelian.
However, $G$ is not $p'$-by-cyclic $p$ with $p=13$.
Thus, by Theorem~\ref{thm:MainThmA}, $K[G]$ is not a K\"{o}the ring.
\end{Examp}


\end{document}